\UseRawInputEncoding
\documentclass[leqno]{amsart}

\usepackage{stmaryrd,graphicx}
\usepackage{amssymb,mathrsfs,amsmath,amscd,amsthm,color}
\usepackage{float}
\usepackage{mathabx}
\usepackage[all,cmtip]{xy}
\DeclareMathAlphabet{\mathpzc}{OT1}{pzc}{m}{it}
\usepackage{amsfonts,latexsym,wasysym}

\DeclareMathOperator{\im}{Im}
\newcommand{\rkn}{\mathbf{wrk}_n}
\newcommand{\frkn}{\mathbf{fwrk}_n}
\newcommand{\rk}{\mathbf{wrk}}
\newcommand{\frk}{\mathbf{fwrk}}

\newcommand{\horkn}{\mathbf{hwrk}_n}

\newcommand{\san}{\mathscr{S}}
\newcommand{\wildn}{\mathbf{w}_{n}}
\newcommand{\wild}{\mathbf{w}}
\newcommand{\sw}{\bigcurlyvee}

\newcommand{\ds}{\displaystyle}
\newcommand{\ui}{[0,1]}
\renewcommand{\int}{\text{int}}

\newcommand{\fe}{\mathbb{FE}}
\newcommand{\fwildn}{\mathbf{fw}_{n}}
\newcommand{\fwild}{\mathbf{fw}}

\newcommand{\scrb}{\mathscr{B}}

\newcommand{\scrs}{\mathscr{S}}
\newcommand{\scrv}{\mathscr{V}}

\newcommand{\bbe}{\mathbb{E}}
\newcommand{\bbn}{\mathbb{N}}

\newcommand{\bbr}{\mathbb{R}}

\newcommand{\bbz}{\mathbb{Z}}

\newcommand{\ov}{\overline}

\newtheorem{theorem}{Theorem}[section]
\newtheorem{lemma}[theorem]{Lemma}
\newtheorem{proposition}[theorem]{Proposition}
\newtheorem{corollary}[theorem]{Corollary}
\theoremstyle{definition}\newtheorem{definition}[theorem]{Definition}
\newtheorem{example}[theorem]{Example}

\newtheorem{remark}[theorem]{Remark}
\newtheorem{problem}[theorem]{Problem}

\begin{document}
\title[Transfinitely iterated wild sets]{Transfinitely iterated wild sets}

\author[J. Brazas]{Jeremy Brazas}
\address{West Chester University\\ Department of Mathematics\\
West Chester, PA 19383, USA}
\email{jbrazas@wcupa.edu}

\author[A. Mitra]{Atish Mitra}
\address{Montana Technical University\\ Department of Mathematical Sciences\\
1300 West Park Street Butte, MT 59701, USA}
\email{amitra@mtech.edu}

\date{\today}

\begin{abstract}
In this paper, we study homotopical analogues of the Cantor-Bendixson derivative. For each $n\geq 0$, the \textit{$\pi_n$-wild set} $\wildn(X)$ of a topological space $X$ is the subspace of $X$ consisting of the points at which there exists a shrinking sequence of essential based maps $S^n\to X$. Since the operator $\wildn$ permits iteration, every given space $X$ yields a descending transfinite sequence of nested subspaces $\{\wildn^{\kappa}(X)\}_{\kappa}$ that stabilizes at some smallest ordinal $\rkn(X)$ called the \textit{$\pi_n$-wild rank} of $X$. We show that the entire transfinite sequence $\{ho(\wildn^{\kappa}(X))\}_{\kappa}$ of homotopy types is a homotopy invariant of $X$ and that $\rkn(X)$ can be an arbitrary countable ordinal when $X$ is an $n$-dimensional Peano continuum. It remains open if there exists a continuum $X$ with uncountable $\pi_n$-wild rank. This difficulty motivates the parallel study a basepoint-free version $\frkn(X)$, called the \textit{free $\pi_n$-wild rank} of $X$. We show that for every continuum $X$, $\frkn(X)$ is always countable and can be any countable ordinal.
\end{abstract}

\maketitle

\section{Introduction}

When considering the homotopical structure of locally complicated spaces such as fractals and Peano continua, the existence of a ``wild point" limits the applicability of traditional methods. Ironically, it is precisely these points and their relationship to infinite products in homotopy groups that have been used to distinguish and classify homotopy types of Peano continua \cite{connerkent,ConnerMeilstrup,EdaSpatial}. For $n\geq 0$, the \textit{$\pi_n$-wild set} $\wildn(X)$ of a topological space $X$ is the subspace of $X$ consisting of points $x\in X$ at where there exists a sequence $f_k:(S^n,\ast)\to (X,x)$ of essential based maps that converges to the constant map at $x$ (in the compact-open topology). When $X$ is first countable and $n=1$, $\wildn(X)$ is the set of points at which $X$ fails to be semilocally simply connected. In particular, the $\pi_1$-wild set plays a key role in the homotopy theory of one-dimensional \cite{Edaonedim,Meilstrup} and planar spaces \cite{Kentplanar} due to its invariance under homotopy equivalences \cite[Section 9]{BFpants}.

Higher dimensional wild sets $\wildn(X)$, $n\geq 2$ have recently been shown to enjoy similar invariance properties \cite{BrazasMitraHigher} and are expected to play an important role as progress is made in higher dimensions. The main problem of interest here beyond computing higher homotopy groups is whether or not weakly homotopy equivalent finite dimensional Peano continua must be homotopy equivalent \cite[Problem 1.1]{BrazasMitraHigher} (see \cite[Problem 4.5]{edakarrep} and \cite[Problem 5.1]{krnoncontractible} for a weaker question).

Since $\wildn(X)$ is a space in its own right, $\wild_n(\wild_n(X))$ may be non-empty, i.e. a wild set may have its own internal wildness. This possibility allows one to recursively define the iterated wild sets $\wild_{n}^{k}(X)=\wildn(\wildn^{k-1}(X))$ to obtain a descending sequence of subspaces $X\supseteq \wildn(X)\supseteq \wildn^{2}(X)\supseteq \wildn^3(X)\supseteq $. Motivated by an early version of the current paper, the initial sequence $X,\wild_1(X),\wild_{1}^{2}(X),\wild_{1}^{3}(X),\dots$ in dimension $n=1$ was applied successfully in \cite{BrazasPavesic} to compute the topological complexity of many one-dimensional Peano continua. Furthermore, it is possible that the intersection $\bigcap_{k\in\omega}\wildn^{k}(X)$ of the sequence of iterated wild sets is non-empty and itself has a non-trivial $\pi_{n}$-wild set. By defining $\wildn^{\omega}(X)=\bigcap_{k\in\omega}\wildn^{k}(X)$ and continuing the definition with transfinite recursion, we define $\wildn^{\kappa}(X)$ for any ordinal $\kappa$. Basic set-theoretic considerations ensure that the transfinite sequence $\{\wildn^{\kappa}(X)\}_{\kappa}$ stabilizes at some smallest ordinal $\rkn(X)=\kappa$, which we call the \textit{$\pi_n$-wild rank of $X$}.

It is instructive to compare $\rkn(X)$ with the classical Cantor-Bendixson rank of a space $X$. Recall that the \textit{Cantor-Bendixson derivative} of a space $X$ is the subspace $L(X)$ of $X$ consisting of all limit points (i.e. non-isolated points) of $X$. Define $L^{0}(X)=X$, $L^{\kappa+1}(X)=L(L^{\kappa}(X))$ for all ordinals $\kappa$ and $L^{\kappa}(X)=\bigcap_{\lambda<\kappa}L^{\lambda}(X)$ when $\kappa$ is a limit ordinal. The transfinite sequence $\{L^{\kappa}(X)\}_{\kappa}$ stabilizes at either an empty space or a perfect set and the smallest ordinal $\kappa$ for which $L^{\kappa+1}(X)=L^{\kappa}(X)$ is the \textit{Cantor-Bendixson rank} of $X$. In comparison, the points of $\wildn(X)$ are ``limit points" of $X$ in a homotopical sense and the transfinite sequence $\{\wildn^{\kappa}(X)\}_{\kappa}$ will stabilize either at an empty space or a perfectly $\pi_n$-wild space, i.e. one which is wild at all of its points (Definition \ref{defperfectlywild}).

In this paper, we show that a homotopy equivalence $X\simeq Y$, restricts to a homotopy equivalence $\wildn^{\kappa}(X)\simeq \wildn^{\kappa}(X)$ for all ordinals $\kappa$ (Theorem \ref{homotopyinvariance2}). Hence, if $ho(X)$ denotes the homotopy type of a space $X$, then the entire transfinite sequence $\{ho(\wildn^{\kappa}(X))\}_{\kappa}$ of homotopy types is itself a homotopy invariant of $X$ (Corollary \ref{invariancecorollary}). Our main result is the following theorem, which shows that for a Peano continuum $X$, the transfinite sequence $\{\wildn^{\kappa}(X)\}_{\kappa}$ may continue non-trivially into arbitrarily large countable ordinals. 

\begin{theorem}\label{thm1}
For any countable ordinal $\lambda\geq 1$, there exists a based $n$-dimensional Peano continuum $(X_{\lambda},x_{\lambda})$ such that $\rkn(X_{\lambda})=\lambda$ and $\wild_{n}^{\lambda}(X_{\lambda})=\emptyset$. Moreover, when $\lambda\geq 2$ is a successor ordinal, we may arrange for $\wild_{n}^{\lambda-1}(X_{\lambda})=\{x_{\lambda}\}$.
\end{theorem}

Some technical issues are created by the following combination of facts: If $X$ is not locally path-connected, then $\wildn(X)$ need not be closed in $X$ \cite[Example 2.12]{BrazasMitraHigher}. Additionally, if $X$ is a Peano continuum, then $\wildn(X)$ can be an arbitrary continuum (compact metrizable space) \cite[Theorem 1.2]{BrazasMitraHigher}. When $X$ is a non-locally path-connected continuum, $\wildn^{\kappa}(X)$ need not be closed in $X$ and it becomes unclear if $\{\wildn^{\kappa}(X)\}_{\kappa}$ stabilizes at a countable ordinal.

Hence, $\wildn^{\kappa}(X)$ need not be closed in $X$ for $\kappa\geq 2$. These difficulties highlight the non-triviality of the following open problems.

\begin{problem}\label{openprob1}
Let $n\geq 1$. Characterize the class of spaces which are homeomorphic to $\wildn(X)$ for a continuum $X$.
\end{problem}

\begin{problem}\label{openprob2}
Let $n\geq 1$. Is there a continuum $X$ (equivalently, a Peano continuum) such that $\rkn(X)$ is uncountable?
\end{problem}

It seems likely that the answer to both of these problems is independent of $n$. The equivalence in Problem \ref{openprob2} is due to \cite[Theorem 1.2]{BrazasMitraHigher}, which is stated as Theorem \ref{realizingcompactmetricspacesthm} in the current paper.

In \cite{FRVZ11}, an important difference between the unbased and based versions of the \textit{semilocally simply connected} property is considered. This same difference is to blame for the difficulty in answering Problems \ref{openprob1} and \ref{openprob2}. While $\wildn(X)$ is the notion of ``wild set" that is most relevant to the algebra of homotopy groups we are motivated to also consider a natural basepoint-free analogue of $\wildn(X)$. For $n\geq 0$, \textit{free $\pi_n$-wild set} $\fwildn(X)$ of a topological space $X$ is the subspace of $X$ consisting of points $x\in X$ at where there exists a sequence $f_k:S^n\to X$ of essential \textit{unbased} maps that converges to the constant map at $x$. We obtain an analogous transfinite sequence $\{\fwildn^{\kappa}(X)\}_{\lambda}$ and rank $\frkn(X)$ that even more closely behaves like the classical Cantor-Bendixson construction. In particular, for a continuum $X$, $\{\fwildn^{\kappa}(X)\}_{\lambda}$ is a descending sequence of \textit{closed} subspaces of $X$ and $\frkn(X)$ must be countable for elementary reasons. Our main result regarding free $\pi_n$-wild sets, which is analogous to Theorem \ref{thm1}, shows that every countable ordinal is a free $\pi_n$-wild rank.

\begin{theorem}\label{thm2}
If $X$ is a continuum, then $\frkn(X)$ is a countable ordinal. Conversely, for any countable ordinal $\lambda\geq 0$, there exists a continuum $X_{\lambda}$ such that $\frkn(X_{\lambda})=\lambda$ and such that $\fwild_{n}^{\lambda-1}(X_{\lambda})=S^n$ whenever $\lambda$ is a successor ordinal.
\end{theorem}

In both Theorem \ref{thm1} and Theorem \ref{thm2}, the respective condition for the successor ordinal case is present to aid in the proof, which is by transfinite induction.

\section{Notation and Preliminaries}

All topological spaces are assumed to be Hausdorff and a ``map" is a continuous function. A continuum is a compact metrizable space and a Peano continuum is a path-connected and locally path-connected continuum. We use Lebesgue covering dimension for topological dimension and refer to \cite{EngelkingDimThry} for standard dimension theory.

For $n\geq 0$, $S^n=\{(x_1,x_2,\dots,x_{n+1})\in\bbr^{n+1}\mid \sum_{i}x_{i}^{2}=1\}$ will be the unit $n$-sphere with basepoint $s_0=(1,0,\dots,0)$. A map $f:S^n\to X$ is said to be \textit{inessential} if it is null-homotopic and \textit{essential} otherwise. When $X$ is an unbased space, we let $X^{+}=X\sqcup\{x_0\}$ denote the space with an additional isolated point.


When $X$ and $Y$ are topological spaces and we write $ho(X)$ to denote the homotopy type of $X$ and $Y^X$ will denote the space of continuous functions $X\to Y$ with the compact-open topology. If $A\subseteq X$ and $B\subseteq Y$, then $(Y,B)^{(X,A)}$ denotes the subspace of $Y^X$ consisting of relative maps $(X,A)\to (Y,B)$. When $y\in Y$, $c_y:X\to Y$ will denote the constant map at $y$. For a based topological space $(X,x_0)$, we write $\Omega^{n}(X,x_0)$ to denote the $n$-loop space $(X,x_0)^{(S^n,s_0)}$ and $\pi_n(X,x_0)=\{[f]\mid f\in \Omega^n(X,x_0)\}$ to denote the $n$-th homotopy group. If $f\in \Omega^n(X,x_1)$ and $\alpha:\ui\to X$ is a path from $x_0$ to $x_1$, then $\alpha\ast f\in \Omega^n(X,x_0)$ will denote the standard path-conjugate of $f$ by $\alpha$. 

We say that a sequence $\{f_k\}_{k\in \bbn}$ of maps $f_k:X\to Y$ \textit{converges to} $y\in Y$ if $\{f_k\}_{k\in\bbn}\to c_y$ in $Y^X$, that is, if for every neighborhood $U$ of $y$, there exists $K\in\bbn$ such that $\im(f_k)\subseteq U$ for all $k\geq K$. 

\begin{definition}\label{wildsetdef}
Let $n\geq 0$. 
\begin{enumerate}
\item A point $x\in X$ is a \textit{$\pi_n$-wild point of $X$} if there exists a sequence $\{f_k\}_{k\in\bbn}$ of essential based maps $f_k:(S^n,s_0)\to (X,x)$ that converges to $x$. Let $\wildn(X)$ denote the subspace of $X$ consisting of all $\pi_n$-wild points of $X$.
\item A point $x\in X$ is a \textit{free $\pi_n$-wild point} if there exists a sequence $\{f_k\}_{k\in\bbn}$ of essential unbased maps $f_k:S^n\to X$ that converges to $x$. Let $\fwildn(X)$ denote the subspace of $X$ consisting of all free $\pi_n$-wild points of $X$.
\end{enumerate}
\end{definition}

\begin{definition}
The \textit{shrinking wedge} of countable set $\{(A_j,a_j)\}_{j\in J}$ of based spaces is the space $\sw_{j\in J}(A_j,a_j)$ whose underlying set is the usual one-point union $\bigvee_{j\in J}(A_j,a_j)$ with canonical basepoint $b_0$ and where $A_j$ is identified canonically as a subset. A set $U$ is open in $\sw_{j\in J}A_j$ if
\begin{itemize}
\item $U\cap A_j$ is open in $A_j$ for all $j\in J$,
\item and whenever $b_0\in U$, we have $A_j\subseteq U$ for all but finitely many $j\in J$.
\end{itemize}
When the basepoints and/or indexing set are clear from context, we may write the shrinking wedge as $\sw_{J}A_j$. 
\end{definition}

\begin{example}[Based infinite earrings]
The \textit{$n$-dimensional infinite earring space} is the shrinking wedge $\bbe_n=\sw_{j\in\bbn}S^n$ of $n$-spheres with canonical basepoint $b_0$ (See Figure \ref{fig1}). We identify $\bbe_0=\sw_{\bbn}(S^0,1)$ with the space $\{1,1/2,1/3,\dots,0\}$ consisting of a single convergent sequence and basepoint $b_0=0$. Let $\ell_j:S^n\to \bbe_n$ denote the inclusion of the $j$-th sphere. When $n\geq 2$, it is known that $\bbe_n$ is $(n-1)$-connected, locally $(n-1)$-connected and that the canonical map $\Psi_n:\pi_n(\bbe_n)\to \check{\pi}_{n}(X)\cong\prod_{j\in\bbn}\bbz$ to the $n$-th shape homotopy group is an isomorphism \cite{EK00higher}. Moreover, $\bbe_n$ has non-trivial singular homology in arbitrarily high dimensions \cite{BarrattMilnor}.
\end{example}

\begin{example}[Free infinite earrings]
The free $n$-dimensional infinite earring space is the shrinking wedge $\fe_n=\sw_{j\in\bbn}(S^n)^{+}$ (where the basepoint of $(S^n)^+$ is the added isolated basepoint). Note that $\fe_n$ (with any choice of metric) consists of a sequence of $n$-spheres of null-diameter that shrinking two a basepoint $a_0$ that lies in its own path component (See Figure \ref{fig1}). Let $\mu:S^n\to \fe_n$ denote the inclusion of the $j$-th sphere.
\end{example}

Note that $\bbe_0\cong \fe_0$.

\begin{figure}[H]
\centering \includegraphics[height=1.3in]{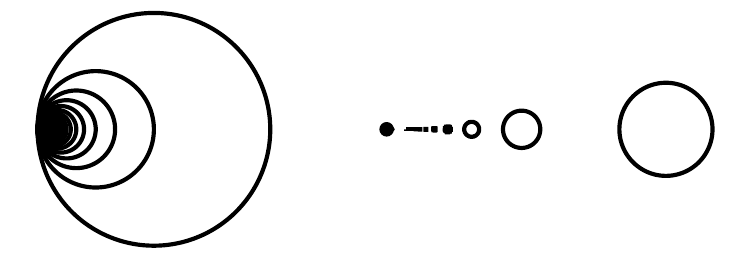}
\caption{\label{fig1} The based $1$-dimensional earring $\bbe_1$ (left) and the free infinite earring $\fe_1$ (right). In the higher-dimensional analogues, circles are replaced by higher-dimensional spheres.}
\end{figure}

\begin{definition}
For a map $f:\bbe_n\to X$, we will refer to $f_j=f\circ \ell_j$ as the \textit{$j$-th restriction} of $f$. We say that a map $f:\bbe_n\to X$ is \textit{fully essential} if the $j$-th restriction $f_j:S^n\to X$ is essential for all $j\in\bbn$.

Similarly, for a map $f:\fe_n\to X$, we will refer to $f_j=f\circ \mu_j$ as the \textit{$j$-th restriction} of $f$. We say that a map $f:\fe_n\to X$ is \textit{fully essential} if the $j$-th restriction $f_j:S^n\to X$ is essential for all $j\in\bbn$.
\end{definition}

The following is straightforward consequence of the shrinking wedge topology that we will apply without reference in the remainder of the paper.

\begin{lemma}\label{felemma}
Let $x\in X$ and $n\geq 0$. Then 
\begin{enumerate}
\item $x\in \wildn(X)$ if and only if there exists a fully essential based map $f:(\bbe_n,b_0)\to (X,x)$.
\item $x\in \fwildn(X)$ if and only if there exists a fully essential based map $f:(\fe_n,a_0)\to (X,x)$.
\end{enumerate}
\end{lemma}

\begin{remark}
The inclusion $\wildn(X)\subseteq \fwildn(X)$ holds for an arbitrary space $X$. Indeed, the canonical map $\iota:\fe_n\to \bbe_n$, which maps $a_0$ to $b_0$ and which maps the $j$-th spheres of $\fe_n$ homeomorphically to the $j$-th sphere of $\bbe_n$ is fully essential. Therefore, if $f:(\bbe_n,b_0)\to (X,x)$ is fully, essential, then so is $f\circ\iota:(\fe_n,a_0)\to (X,x)$.

In general, $\wildn(X)$ can be a proper subset of $\fwildn(X)$. For example, we have $\fwildn(\fe_n)=\{a_0\}$ and $\wildn(\fe_n)=\emptyset$ since the path components of $\fe_n$ are locally contractible.
\end{remark}

\begin{proposition}\label{zeroequalityprop}
For any space $X$, $\wild_0(X)=\fwild_{0}(X)$.
\end{proposition}

\begin{proof}
One inclusion follows from the previous remark. Suppose $x\in \fwild_0(X)$. Then there exists a fully essential map $f:(\fe_0,a_0)\to (X,x)$. Let $P_j=\{a_j,b_j\}$ denote the $j$-th copy of $S^0$ in $\fe_0$. Then $f(a_j)$ and $f(b_j)$ lie in distinct path-components of $X$. In particular, for each $j$, at least one of $f(a_j)$ or $f(b_j)$ does not lie in the path-component of $x$. Relabeling some $a_j$ and $b_j$, if necessary, we may assume that $f(a_j)$ does not lie in the path-component of $x$ for all $j\geq 1$. Define $g:\bbe_0\to X$ so that $g(0)=x$, $g(1/j)=f(a_j)$ for each $j\geq 1$. Then $g$ is fully essential, proving $x\in \wild_0(X)$.
\end{proof}

\begin{lemma}\label{comparelemma}
Let $n\geq 0$. If $X$ is first countable, then $\fwildn(X)$ is closed in $X$. Moreover, if $n\geq 1$ and $X$ is also locally path-connected, then $\wildn(X)=\fwildn(X)$.
\end{lemma}

\begin{proof}
Suppose $\{x_m\}_{m\in\bbn}\to x$ is a convergent sequence in $X$ where $x_m\in \fwildn(X)$ for all $m$. Let $U_1\supseteq U_2\supseteq U_3\supseteq \cdots$ be a neighborhood base at $x$. Find $m_1<m_2<m_3<\cdots$ such that $x_{m_k}\in U_k$. Since $x_{m_k}\in \fwildn(X)$, there exists an essential map $f_k:S^n\to X$ with $\im(f_k)\subseteq U_k$. Now $\{f_k\}_{k\in\bbn}$ is a sequence of essential maps that converges to $x$, giving $x\in \fwildn(X)$. For the second statement, recall that $\wildn(X)\subseteq \fwildn(X)$ holds in general. Suppose $X$ is locally path-connected and $x\in \fwildn(X)$. Let $f:(\fe_n,a_0)\to (X,x)$ be a fully essential map. Let $x_j$ be the image of $s_0$ under the $j$-th restriction $f_j:S^n\to X$ of $f$. Since $X$ is locally path-connected and first countable at $x$, there exists a sequence $\{\alpha_j\}_{j\in\bbn}$ of paths $\alpha_j:\ui\to X$ that converges to $x$ and where $\alpha_j(0)=x$ and $\alpha_j(1)=x_j$ for each $j$. Let $g_j:(S^n,s_0)\to (X,x)$ be the path-conjugate of $f_j$ by $\alpha_j$. Now $\{g_j\}_{j\in\bbn}$ is a sequence of essential based maps that converges to $x$, giving $x\in \wildn(X)$.
\end{proof}

\begin{example}
The space constructed in \cite[Example 2.11]{BrazasMitraHigher} is a locally path-connected but non-first countable space $X$ for which $\wildn(X)=\fwildn(X)$ is not closed in $X$. The space constructed in \cite[Example 2.12]{BrazasMitraHigher} is a compact but non-locally path-connected metric space where $\wildn(X)$ is not closed in $X$ (in particular, $\fwildn(X)$ is the closed topologists sine curve and $\wildn(X)$ is the non-closed path-component of $\fwildn(X)$).
\end{example}

The next results are proved in \cite{BrazasMitraHigher} for based $\pi_n$-wild sets. The proof for free $\pi_n$-wild sets are essentially the same, replacing fully essential maps on $\bbe_n$ with fully essential maps on $\fe_n$.

\begin{lemma}\label{inductionsteplemma}\cite[Lemma 4.2]{BrazasMitraHigher}
If $f:X\to Y$ is $\pi_n$-injective, then $f(\wild_{n}(X))\subseteq \wild_{n}(Y)$ and $f(\fwildn(X))\subseteq \fwildn(Y)$. Moreover, any (free) homotopy $H:X\times \ui\to Y$ between $\pi_n$-injective maps $f,g:X\to Y$, restricts to a homotopy 
\begin{enumerate}
\item $\wild_{n}(X)\times \ui\to \wild_n(Y)$ between maps $f|_{\wild_{n}(X)},g|_{\wild_{n}(X)}:\wild_{n}(X)\to \wild_{n}(Y)$
\item $\fwild_{n}(X)\times \ui\to \fwild_n(Y)$ between maps $f|_{\fwild_{n}(X)},g|_{\fwild_{n}(X)}:\fwild_{n}(X)\to \fwild_{n}(Y)$.
\end{enumerate}
\end{lemma}

\begin{theorem}\label{homotopyinvariance}\cite[Theorem 4.6]{BrazasMitraHigher}
For each $n\geq 1$, the homotopy types of based and free $\pi_n$-wild sets are homotopy invariants. In particular, if $f:X\to Y$ is a homotopy equivalence, then $f|_{\wildn(X)}:\wildn(X)\to \wildn(Y)$ and $f|_{\fwildn(X)}:\fwildn(X)\to \fwildn(Y)$ are homotopy equivalences.
\end{theorem}

\begin{lemma}\label{disjointunionlemma}\cite[Proposition 3.5]{BrazasMitraHigher}
For any $n\geq 0$ and collection of spaces $\{X_{\lambda}\}_{\lambda}$, we have \[\wildn\left(\coprod_{\lambda}X_{\lambda}\right)=\coprod_{\lambda}\wildn(X_{\lambda})\text{ and }\fwildn\left(\coprod_{\lambda}X_{\lambda}\right)=\coprod_{\lambda}\fwildn(X_{\lambda}).\]
\end{lemma}

\begin{theorem}\label{realizingcompactmetricspacesthm}\cite[Theorem 1.2]{BrazasMitraHigher}
If $X$ is a Peano continuum, then $\wildn(X)$ is a continuum. Conversely, if $Y$ is a continuum, then there exists a Peano continuum $X$ such that
\begin{enumerate}
\item $Y\subseteq X$ and $X\backslash Y$ is a countable disjoint union of open $1$-cells and open $n$-cells.
\item $\wildn(X)=Y$,
\item $\dim(X)= \max\{n,\dim(Y)\}$.
\end{enumerate}
\end{theorem}

\begin{example}
For general continua, the sets $\fwild_n(\wild_n(X))$ and $\wild_n(\fwild_n(X))$ are not comparable. Using Theorem \ref{realizingcompactmetricspacesthm}, we can construct a Peano continuum $X_1$ for which $\wild_n(X_1)=\fwild_n(X_1)=\fe_n$. In this case, $\fwild_n(\wild_n(X_1))=\{a_0\}$ and $\wild_n(\fwild_n(X_1))=\emptyset$. Next, let $Y$ be a countably infinite disjoint union of $n$-spheres in $\bbr^{n+1}\backslash \bbe_n$ of null-diameter such that $\ov{Y}=Y\cup \bbe_n$. Set $X_2=\ov{Y}$. In this case, we have $\fwildn(\wildn(X_2))=\emptyset$ and $\wild_n(\fwildn(X_2))=\{b_0\}$. If we define $X=X_1\sqcup X_2$, then $\fwild_n(\wild_n(X))=\{a_0\}$ and $\wild_n(\fwildn(X))=\{b_0\}$.
\end{example}

\begin{definition}\label{defperfectlywild}
A space $X$ is \textit{perfectly (resp. free) $\pi_n$-wild} if $\wildn(X)=X$ (resp. $\fwildn(X)=X$).
\end{definition}

The ternary Cantor set is perfectly $\pi_0$-wild. Well-known examples of perfectly $\pi_1$-wild spaces include the Sierpinski Carpet, Menger Curve, and others \cite{ConnerEda,ConnerEda2}. Examples of perfectly $\pi_n$-wild spaces for $n\geq 2$ are constructed in \cite{BrazasMitraHigher}. Since $\wildn(X)\subseteq \fwildn(X)$ holds in general, every perfectly $\pi_n$-wild space is perfectly free $\pi_n$-wild. 

\section{Iterating the $\pi_n$-wild set operator}

For any space $X$, we recursively define a descending transfinite sequence of subspaces of $X$: Let $\wild_{n}^{0}(X)=\fwildn^{0}(X)=X$. Set $\wild_{n}^{\kappa+1}(X)=\wild_n(\wild_{n}^{\kappa}(X))$ and $\fwildn^{\kappa+1}(X)=\fwildn(\fwildn^{\kappa}(X))$ for each ordinal $\kappa\geq 0$. Define $\wild_{n}^{\kappa}(X)=\bigcap_{\lambda<\kappa}\wild_{n}^{\kappa}(X)$ and $\fwildn^{\kappa}(X)=\bigcap_{\lambda<\kappa}\fwildn^{\kappa}(X)$ if $\kappa$ is a limit ordinal. We refer to $\wild_{n}^{\kappa}(X)$ as the \textit{$\kappa$-th iterated $\pi_n$-wild set of $X$} and $\fwildn^{\kappa}(X)$ as the \textit{$\kappa$-th iterated free-$\pi_n$-wild set of $X$}.

Note that the transfinite sequences $\{\wild_{n}^{\kappa}(X)\}_{\kappa}$ and $\{\fwildn^{\kappa}(X)\}_{\kappa}$ are ``descending" in the sense that $\wild_{n}^{\kappa}(X)\subseteq \wild_{n}^{\lambda}(X)$ and $\fwildn^{\kappa}(X)\subseteq \fwildn^{\lambda}(X)$ if $\kappa\geq \lambda$. Since $X$ is a fixed space both of these transfinite sequences must stabilize, i.e. there must be an ordinal $\kappa_0$ for which $\wild_{n}^{\kappa_0}(X)=\wild_{n}^{\kappa_0+1}(X)$ and an ordinal $\kappa_1$ for which $\wild_{n}^{\kappa_1}(X)=\wild_{n}^{\kappa_1+1}(X)$ (otherwise, one could inject an arbitrarily large ordinals into the power set of $X$). 

\begin{definition}
Let $X$ be a space. 
\begin{enumerate}
\item The smallest ordinal $\kappa_0$ such that $\wild_{n}^{\kappa_0}(X)=\wild_{n}^{\kappa}(X)$ for all $\kappa\geq \kappa_0$ is the \textit{(based) $\pi_n$-wild rank of $X$} and is denoted $\rkn(X)$.
\item The smallest ordinal $\kappa_1$ such that $\fwildn^{\kappa_1}(X)=\fwildn^{\kappa}(X)$ for all $\kappa\geq \kappa_1$ is the \textit{free $\pi_n$-wild rank of $X$} and is denoted $\frkn(X)$.
\end{enumerate}
\end{definition}

\begin{remark}
Note that $\rkn(X)=\kappa_0$ (resp. $\frkn(X)=\kappa_1$) is the smallest ordinal for which the inclusion $\wildn^{\kappa_0+1}(X)\to\wildn^{\kappa_0}(X)$ (resp. $\fwildn^{\kappa_1+1}(X)\to\fwildn^{\kappa_1}(X)$) is the identity map. In particular, the transfinite sequence $\{\wild_{n}^{\kappa}(X)\}_{\kappa}$ either stabilizes in the empty set or at a perfectly $\pi_n$-wild set and $\{\fwild_{n}^{\kappa}(X)\}_{\kappa}$ either stabilizes in the empty set or at a perfectly free $\pi_n$-wild set.
\end{remark}

\begin{remark}\label{zerodimcaseremark}
In dimension $n=0$, it follows from Proposition \ref{zeroequalityprop} that for any space $X$, $\wild_{0}^{\kappa}(X)=\fwild_{0}^{\kappa}(X)$ for all ordinals $\kappa$. Thus $\rk_0(X)=\frk_0(X)$.
\end{remark}

\begin{example}
If $\wildn(X)$ is finite and non-empty, then $\rkn(X)=2$. For example, since $\wildn(\bbe_n)=\{b_0\}$, we have $\rkn(\bbe_n)=2$.
\end{example}

\begin{example}[The wild circle and its variants]\label{wildcircleexample}
Consider the \textit{wild circle} $X\subseteq \bbr^2$ from \cite{BrazasMitraHigher}, where a shrinking sequence of circles are attached to $S^1$ in a dense fashion so that $X$ is a one-dimensional planar Peano continuum with $\wild_1(X)=S^1$. Take the basepoint of $X$ to be a point $x_0$ that lies in $S^1$ and let $Y=\sw_{j\in \bbn}X$ be a shrinking wedge of a sequence of copies of $X$ with wedgepoint $y_0$. Then $Y$ is also a one-dimensional Peano continuum, which can be embedded in the plane. We have $\wild_{1}(Y)\cong \bbe_1$, $\wild_{1}^{2}(Y)=\{y_0\}$, and $\wild_{1}^{3}(Y)=\emptyset$ giving $\rk_1(Y)=3$. By replacing $S^1$ with $S^n$ and $\bbe_1$ with $\bbe_n$, one obtains an analogous example for the higher-dimensional ranks.
\end{example}

Extending, the previous example, repeated application of Theorem \ref{realizingcompactmetricspacesthm} implies that for any $n$-dimensional compact metric space $Y$, there exists an $n$-dimensional Peano continuum $X$ and finite ordinal $\kappa\in\omega$ such that $\wild_{n}^{\kappa}(X)=Y$ and such that $\wild_{n}^{\lambda}(X)$ is a Peano continuum for all $0\leq \lambda\leq \kappa-1$. Taking $Y$ to be a one-point space, we have the following.

\begin{corollary}
For every $\kappa\in\omega$, there exists an $n$-dimensional Peano continuum $X$ with $\rkn(X)=\kappa$.
\end{corollary}

\begin{example}[A rank $\omega+1$ example]\label{omegaplusoneexample}
Let $X_1=\{\ast\}$ and $X_2=\bbe_n$. Using Theorem \ref{realizingcompactmetricspacesthm}, we can construct an $n$-dimensional Peano continuum $X_3$ where $X_3\backslash X_2$ is a countable union of open $n$-cells and $\wildn(X_3)=X_2$. Moreover, we may repeat this construction to obtain a sequence of $n$-dimensional Peano continua $\bbe_n=X_2\subseteq X_3\subseteq X_4\subseteq X_5\subseteq \cdots$ where $\wildn(X_{k+1})=X_k$, $\rkn(X_k)=k$ and $X_{k+1}$ is obtained from $X_k$ by attaching a countable collection of shrinking copies of $\bbe_n$ along a dense set. Let $X=\sw_{k\in\bbn}X_k$ be the shrinking wedge with wedgepoint $x_0$ and note that $\wildn(X)=\sw_{m\geq 2}X_m$. Note that $\wildn(X)\cong X$ although the inclusion $\wildn(X)\to X$ is not a homotopy equivalence. Iterating the wild set gives $\wildn^{k}(X)=\sw_{m\geq k+1}X_m$. Since $\{x_0\}=\bigcap_{k\in\omega}\wildn^{k}(X)$, we have $\wildn^{\omega+1}(X)=\emptyset$ and thus $\rkn(X)=\omega+1$.
\end{example}

\begin{example}[A rank $\omega$ example]\label{omegaexample}
We can modify the previous example slightly to obtain an $n$-dimensional Peano continuum with $\pi_n$-wild rank $\omega$ and whose sequence of iterated wild sets terminates at the empty set. As before, take $X_1=\{\ast\}$ and recursively apply Theorem \ref{realizingcompactmetricspacesthm} to construct $n$-dimensional Peano continua $X_2,X_3,X_4,\dots$ so that $\wildn(X_k)=X_{k-1}$ for all $k\geq 2$. Let $x_k$ be a given basepoint in $X_k$ and let $Y_k=(X_k,x_k)\vee (\ui,0)$ be the space given by attaching a ``whisker" to $X_k$. We identify $X_k$ naturally as a subspace of $Y_k$ and let $y_k$ denote the image of $1$ (end of the whisker) in $Y_k$. Let $Y=\sw_{k\in\bbn}(Y_k,y_k)$ with wedgepoint $y_0$ and note that $Y$ is an $n$-dimensional Peano continuum. We have $\wildn(Y)=\{y_0\}\cup \bigcup_{k\geq 2}\wild_n(X_k)$. However, $y_0$ will not lie in the second $\pi_n$-wild set since $\{y_0\}$ is a simply connected path component of $\wildn(Y)$. Thus $\wild_{n}^{k}(Y)=\bigcup_{m\geq k+1}\wild_n(X_m)$ for all $k\geq 2$. Since $\wildn^{\omega}(Y)=\bigcap_{k\in\omega}\wild_{n}^{k}(Y)=\emptyset$, we have $\rkn(Y)=\omega$. 
\end{example}

\begin{remark}[Cardinal Bounds]\label{upperbound}
For any space $X$ and $n\geq 0$, the axiom of choice ensures that $|\rkn(X)|\leq |X|$ and $|\frkn(X)|\leq |X|$: if $\rkn(X)=\kappa_0$, choose $x_{\lambda}\in \wildn^{\lambda}(X)\backslash \wildn^{\lambda+1}(X)$ for each $\lambda\in \kappa_0$ and define injection $\kappa_0\to X$ by $\lambda\mapsto x_{\lambda}$. The same argument applies for free $\pi_n$-wild rank. For example, if $X$ is a separable metrizable space, then $\max\{|\rkn(X)|,|\frkn(X)|\}\leq 2^{\aleph_{0}}$.
\end{remark}

An advantage of the free $\pi_n$-wild rank is the following.

\begin{theorem}\label{countablefreerank}
If $X$ is second countable, then $\frkn(X)$ is a countable ordinal.
\end{theorem}

\begin{proof}
It follows from Lemma \ref{comparelemma} that $\{\fwildn^{\lambda}(X)\}_{\lambda}$ is a descending transfinite sequence of closed subspaces of $X$. It is well-known that any such sequence in a second countable space must stabilize at a countable ordinal \cite{KechrisDST}.
\end{proof}

In dimension $n=0$, Theorem \ref{countablefreerank} and Remark \ref{zerodimcaseremark} combine to give the following.

\begin{corollary}
If $X$ is second countable, then $\rk_0(X)$ is countable.
\end{corollary}

\begin{proposition}\label{onedimprop}
If $X$ is one-dimensional, then $\wild_{1}^{\kappa}(X)\subseteq \fwild_{1}^{\kappa}(X)$ for all $\kappa$.
\end{proposition}

\begin{proof}
The proof is by induction and the general inclusion $\wild_1(X)\subseteq \fwild_1(X)$ serves as the initial step. The limit ordinal case is clear so we focus on the successor ordinal case. Suppose that $\wild_{1}^{\kappa}(X)\subseteq \fwild_{1}^{\kappa}(X)$ for some ordinal $\kappa$. If $B$ is a one-dimensional Hausdorff space and $A\subseteq B$ then the inclusion $A\to B$ is $\pi_1$-injective \cite[Corollary 3.3]{CConedim} and thus $\wild_1(A)\subseteq \wild_1(B)$. Applying $\wild_1$ to the inclusion $\wild_{1}^{\kappa}(X)\subseteq \fwild_{1}^{\kappa}(X)$ gives $\wild_1(\wild_{1}^{\kappa}(X))\subseteq \wild_1(\fwild_{1}^{\kappa}(X))$. Since $\wild_1(Z)\subseteq \fwild_1(Z)$ holds in general, we have $\wild_1(\fwild_{1}^{\kappa}(X))\subseteq \fwild_1(\fwild_{1}^{\kappa}(X))$. Combining these inclusions gives $\wild_{1}^{\kappa+1}(X)\subseteq \fwild_{1}^{\kappa+1}(X)$.
\end{proof}

\begin{remark}
When $X$ is a one-dimensional Hausdorff space, Proposition \ref{onedimprop} does not imply a general relationship between the $\rk_1(X)$ and $\frk_1(X)$. However, if $\{\fwild_{1}^{\kappa}(X)\}_{\kappa}$ stabilizes at the empty space, then so does $\{\wild_{1}^{\kappa}(X)\}_{\kappa}$ and we have $\rk_1(X)\leq \frk_1(X)$. 
\end{remark}

\begin{remark}[Rank is not homotopy invariant]
Neither the based $\pi_n$-wild rank nor the free $\pi_n$-wild rank is an invariant of homotopy type. For instance, let $X$ be any perfectly $\pi_n$-wild set, e.g. the Menger curve if $n=1$. Let $Y=X\vee \ui$ consists of $X$ with a ``whisker" attached. Then $\wildn(Y)=X$ and the inclusion $X\to Y$ is a homotopy equivalence. Hence, $X\simeq Y$ but $\rkn(Y)=\frkn(Y)=1\neq 0=\rkn(X)=\frkn(X)$.

Although we do not study it in detail here, it is possible to define a potentially useful version of the $\pi_n$-wild rank that does happen to be homotopy invariant. One might define the \textit{homotopy $\pi_n$-wild rank of $X$} to be the smallest ordinal $\horkn(X)=\kappa_0$ for which the inclusion $\wildn^{\lambda}(X)\to\wildn^{\kappa_0}(X)$ is a homotopy equivalence for all $\lambda>\kappa_0$. It is always the case that $\horkn(X)\leq \rkn(X)$ and it is possible that $\rkn(X)=\horkn(X)+\lambda$ for an arbitrarily large ordinal $\lambda$.
\end{remark}

\begin{lemma}\label{initialsumrank}
If $X$ is a space and $\wildn^{\lambda}(X)\cong A$ for some ordinal $\lambda<\rkn(X)$, then $\rkn(X)=\lambda+\rkn(A)$.
\end{lemma}

\begin{proof}
Identify $A$ with $\wildn^{\lambda}(X)$. The transfinite sequence $\{\wild^{\kappa}(X)\}_{\kappa}$ consists of the initial segment $\{\wild^{\kappa}(X)\}_{\kappa<\lambda}$ followed by $\{\wild^{\kappa}(X)\}_{\lambda\leq\kappa\leq \rkn(X)}$ (whose first term is $A$), which is then followed by the constant terminal segment $\{\wild^{\kappa}(X)\}_{\kappa>\rkn(X)}$. If $\mu$ is an ordinal that is order isomorphic to the ordinal segment $[\lambda,\rkn(X)]$, then $\rkn(A)=\mu$ and $\rkn(X)=\lambda+\mu$. 
\end{proof}

\begin{theorem}\label{homotopyinvariance2}
If $f:X\to Y$ and $g:Y\to X$ are homotopy inverses, then, for any ordinal $\kappa$, so are the restrictions $f_{\kappa}:\wild_{n}^{\kappa}(X)\to \wild_{n}^{\kappa}(Y)$ and $g_{\kappa}:\wild_{n}^{\kappa}(Y)\to \wild_{n}^{\kappa}(X)$
\end{theorem}

\begin{proof}
We proceed by transfinite induction. Let $f_0=f$, $g_0=g$, and fix homotopies $H_0:X\times I\to X$ from $id_X$ to $g\circ f$ and $G_0:Y\times I\to Y$ from $id_Y$ to $f\circ g$. Note that all of the maps $f_0,g_0,H_0,G_0$ are $\pi_n$-injective. Suppose that for all $\lambda<\kappa$, the respective restrictions of $f_0$, $g_0$, $H_0$, and $G_0$ give homotopy inverses $f_{\lambda}:\wild_{n}^{\lambda}(X)\to\wild_{n}^{\lambda}(Y) $ and $g_{\lambda}:\wild_{n}^{\lambda}(Y)\to\wild_{n}^{\lambda}(X) $ with homotopies $H_{\lambda}$ from $id_{\wild_{n}^{\lambda}(X)}$ to $g_{\lambda}\circ f_{\lambda}$ and $G_{\lambda}$ from $id_{\wild_{n}^{\lambda}(Y)}$ to $f_{\lambda}\circ g_{\lambda}$.

If $\kappa=\lambda+1$ is a successor ordinal, we apply Lemma \ref{inductionsteplemma}: let $f_{\kappa}:\wild_{n}^{\kappa}(X)\to \wild_{n}^{\kappa}(Y)$ be the restriction of $f_{\lambda}$ to $\wild_{n}^{\kappa}(X)$ and $g_{\kappa}:\wild_{n}^{\kappa}(Y)\to \wild_{n}^{\kappa}(X)$ be the restriction of $g_{\lambda}$ to $\wild_{n}^{\kappa}(Y)$. Similarly, let $H_{\kappa}$ be the restriction of $H_{\lambda}$ to a map $\wild_{n}^{\kappa}(X)\times I\to \wild_{n}^{\kappa}(X)$ and $G_{\kappa}$ be the restriction of $G_{\lambda}$ to $\wild_{n}^{\kappa}(Y)\times I\to \wild_{n}^{\kappa}(Y)$. Now Lemma \ref{inductionsteplemma} ensures that the homotopies $H_{\kappa}$ and $G_{\kappa}$ give $id_{\wild_{n}^{\kappa}(X)}\simeq g_{\kappa}\circ f_{\kappa}$ and $id_{\wild_{n}^{\kappa}(Y)}\simeq f_{\kappa}\circ g_{\kappa}$ respectively.

If $\kappa$ is a limit ordinal, then $\wild_{n}^{\kappa}(X)=\bigcap_{\lambda<\kappa}\wild_{n}^{\lambda}(X)$ and $\wild_{n}^{\kappa}(Y)=\bigcap_{\lambda<\kappa}\wild_{n}^{\lambda}(Y)$. By our induction hypothesis $f_{\lambda}=f|_{\wild_{n}^{\lambda}(X)}$ and $g_{\lambda}=f|_{\wild_{n}^{\lambda}(Y)}$ for all $\lambda<\kappa$. Therefore, $f_{\kappa}:\wild_{n}^{\kappa}(X)\to \wild_{n}^{\kappa}(Y)$ and $g_{\kappa}:\wild_{n}^{\kappa}(Y)\to \wild_{n}^{\kappa}(X)$ are well-defined. Similarly, $H_{\lambda}=H|_{\wild_{n}^{\lambda}(X)\times \ui}$ for all $\lambda<\kappa$. Therefore, if we take $H_{\kappa}= H|_{\wild_{n}^{\kappa}(X)\times \ui}$, we get a well-defined map $H_{\kappa}$ from \[\wild_{n}^{\kappa}(X)\times \ui=\left(\bigcap_{\lambda<\kappa}\wild_{n}^{\kappa}(X)\right)\times \ui=\bigcap_{\lambda<\kappa}\wild_{n}^{\kappa}(X)\times \ui\]
to $\wild_{n}^{\kappa}(Y)$. It follows that $H_{\kappa}:\wild_{n}^{\kappa}(X)\times \ui\to \wild_{n}^{\kappa}(Y)$ is a homotopy from $id_{\wild_{n}^{\kappa}(X)}$ to $f_{\kappa}\circ g_{\kappa}$. By the same argument, $G_{\kappa}=G|_{\wild_{n}^{\kappa}(Y)\times \ui}$ is a well-defined homotopy from $id_{\wild_{n}^{\kappa}(Y)}$ to $g_{\kappa}\circ f_{\kappa}$.
\end{proof}

\begin{corollary}\label{invariancecorollary}
The transfinite sequence $\{ho(\wild_{n}^{\kappa}(X))\}_{\kappa}$ of homotopy types is a homotopy invariant of $X$.
\end{corollary}

\begin{corollary}
If $n\geq 0$ and $X\simeq Y$, then one of the following hold:
\begin{enumerate}
\item $\{\wild_{n}^{\kappa}(X)\}$ and $\{\wild_{n}^{\kappa}(Y)\}$ both stabilize to $\emptyset$ and $\rkn(X)=\rkn(Y)$,
\item $\{\wild_{n}^{\kappa}(X)\}$ and $\{\wild_{n}^{\kappa}(Y)\}$ both stabilize to homotopy equivalent perfectly $\pi_n$-wild spaces.
\end{enumerate}
\end{corollary}

\begin{proof}
If $f:X\to Y$ and $g:Y\to X$ are homotopy inverses, then Theorem \ref{homotopyinvariance2} implies that the restricted functions $f_{\kappa}:\wild_{n}^{\kappa}(X)\to \wild_{n}^{\kappa}(Y)$ and $g_{\kappa}:\wild_{n}^{\kappa}(Y)\to \wild_{n}^{\kappa}(X)$ are well-defined functions. Hence, $\wild_{n}^{\kappa}(Y)=\emptyset$ if and only if $\wild_{n}^{\kappa}(X)=\emptyset$. It follows that if either one of the transfinite sequences stabilizes to $\emptyset$ at $\rkn(X)=\kappa_0$, then the other sequence also stabilizes to $\emptyset$ at $\kappa_0$. The other possibility is that $\{\wild_{n}^{\kappa}(X)\}$ and $\{\wild_{n}^{\kappa}(Y)\}$ both stabilize to perfectly $\pi_n$-wild spaces (possibly with different ranks), which remain homotopy equivalent following stabilization by Theorem \ref{homotopyinvariance2}.
\end{proof}

\section{Every countable ordinal is a $\pi_n$-wild rank}

In this section, we give constructions to show that for every countable ordinal $\lambda$, there is an $n$-dimensional Peano continuum with $\rkn(X)=\lambda$. The next lemma is straightforward to prove using Proposition \ref{disjointunionlemma}.

\begin{lemma}\label{disjointunionranklemma}
If $n\geq 0$ and $\{X_{i}\}_{i\in I}$ is a collection of spaces, then $\wildn^{\kappa}(\coprod_{i\in I}X_{i})=\coprod_{i\in I}\wildn^{\kappa}(X_i)$ for all ordinals $\kappa$. Moreover, $\rkn\left(\coprod_{i\in I}X_{i}\right)=\sup\{\rkn(X_{i})\mid i\in I\}$.
\end{lemma}


The construction used to prove Theorem \ref{realizingcompactmetricspacesthm} is the following ``shrinking" modification of adjunction spaces, which we will use heavily in the remainder of the paper.

\begin{definition}[Shrinking Point-Attachment Spaces]\label{shrinkingattachmentdef}
Let $X$ be a compact space, $A=\{a_j\}_{j\in\bbn}$ be a sequence (of not necessarily distinct points) in $X$ and let $\scrb=\{(B_j,b_j)\}_{j\in\bbn}$ be a sequence of based spaces. Let $\san(X,A,\scrb)=X\sqcup \coprod_{j\in\bbn}B_j/\mathord{\sim}$ where $a_j\sim b_j$ for all $j\in\bbn$. That is $\san(X,A,\scrb)$ is obtained by attaching each $B_j$ to $X$ by identifying the basepoint of $B_j$ with $a_j$. We give $\san(X,A,\scrb)$ the following topology: $U\subseteq \san(X,A,\scrv)$ is open if and only if
\begin{enumerate}
\item $X\cap U$ is open in $X$
\item $X\cap B_j$ is open in $B_j$ for all $j\in \bbn$,
\item whenever $x\in U\cap X$ and $j_1<j_2<j_3<\cdots$ is such that $\{a_{j_i}\}_{i\in\bbn}\to x$ in $X$, then $B_{j_i}\subseteq U$ for all but finitely many $i\in\bbn$.
\end{enumerate}
When $\scrb=\{B,B,B,\dots\}$ is constant, we write $\san(X,A,B)$ for the shrinking point-attachment.
\end{definition}

\begin{example}
If $X=\{x_0\}$ contains a single point, then $\scrs(X,A,\scrb)=\sw_{j\in\bbn}B_j$.
\end{example}




\begin{proposition}\label{propertiesprop}\cite[Proposition 5.11]{BrazasMitraHigher}
If $X$ and each $B_j$ is separable (resp. path-connected, path connected and locally path-connected), then so is $\san(X,A,\scrb)$.
\end{proposition}

\begin{proposition}\label{peanoprop}\cite[Proposition 5.12]{BrazasMitraHigher}
If $X$ and each $B_j\in\scrb$ is a compact space (respectively, a compact Hausdorff space, continuum, an $n$-dimensional continuum, a Peano continuum, an $n$-dimensional Peano continuum), then so is $\san(X,A,\scrb)$. 
\end{proposition}

Next, we prove the basepoint-free analogue of Theorem \ref{realizingcompactmetricspacesthm}.

\begin{theorem}\label{realizingfreetheorem}
If $X$ is a continuum, then $\fwildn(X)$ is a continuum. Conversely, if $Y$ is a continuum, then there exists a continuum $X$ such that
\begin{enumerate}
\item $Y\subseteq X$ and $X\backslash Y$ is a countable disjoint union of $n$-spheres.
\item $\fwildn(X)=Y$,
\item $\dim(X)= \max\{n,\dim(Y)\}$.
\end{enumerate}
\end{theorem}

\begin{proof}
When $X$ is a continuum, $\fwildn(X)$ is closed in $X$ (Lemma \ref{comparelemma}). Thus $\fwildn(X)$ is a continuum. Given a continuum $Y$, let $A=\{a_j\}_{j\in\bbn}$ be a sequence with dense image in $Y$ and let $\scrb=\{(\fe_n,a_0),(\fe_n,a_0),(\fe_n,a_0),\dots\}$ be the constant sequence of based space. Set $X=\san(Y,A,\scrb)$. Note that $X\backslash Y$ is a disjoint union of countably many $n$-spheres and the equality $\dim(X)= \max\{n,\dim(Y)\}$ is clear. Since $X$ is locally contractible at each point of $X\backslash Y$, we have $\fwildn(X)\subseteq Y$. For each term $a_j$ there is a fully essential map $f_j:(\fe_n,a_0)\to (X,a_j)$ corresponding to the attached copy of $\fe_n$. Thus $\im(A)\subseteq \fwildn(X)$. Since $\fwildn(X)$ is closed in $X$ by Lemma \ref{comparelemma}, we have $Y=\im(A)\subseteq \fwildn(X)$, completing the proof of (2).
\end{proof}

\begin{lemma}\label{sanlemma}\cite[Lemma 5.13]{BrazasMitraHigher}
Suppose $X$ and each $B_j$ is a Peano continuum. If $b_j\in \wildn(B_j)$ for each $j\in\bbn$, then
\[\wildn(\san(X,A,\scrb))=\wildn(X)\cup \ov{\im(A)}\cup \bigcup_{j\in\bbn}\wildn(B_j)\]
Specifically, if $\im(A)$ is dense in $X$, then $\wildn(\san(X,A,\scrb))=X\cup \bigcup_{j\in\bbn}\wildn(B_j)$.
\end{lemma}

In the next lemma, we extend the previous lemma to higher iterations of the $\pi_n$-wild set.

\begin{lemma}\label{swlemma}
If $Y=\san(X,A,\scrb)$ is a shrinking point-attachment with $\scrb=\{(B_j,b_j)\}_{j\in\bbn}$, then for all ordinals $\lambda$, we have
\begin{enumerate}
\item $\ds\bigcup_{j\in\bbn}\wildn^{\lambda}(B_j)\subseteq \wildn^{\lambda}(Y)$
\item and $\wildn^{\lambda}(Y)\backslash X=\coprod_{j\in\bbn}\wildn^{\lambda}(B_j)\backslash\{b_j\}$.
\end{enumerate}
\end{lemma}

\begin{proof}
We prove (1) and (2) simultaneously by transfinite induction on $\lambda$. Certainly both are true for $\lambda=0$. Suppose the two statements hold for all $\kappa<\lambda$. 

If $\lambda$ is a limit ordinal, then for given $j\in\bbn$, we have 
\[
\wildn^{\lambda}(B_j) = \bigcap_{\kappa<\lambda}\wildn^{\kappa}(B_j)\subseteq \bigcap_{\kappa<\lambda}\wildn^{\kappa}(Y)=\wildn^{\lambda}(Y).\]
Thus $\bigcup_{j\in\bbn}\wildn^{\lambda}(B_j)\subseteq \wildn^{\lambda}(Y)$ holds. This inclusion already implies the direction $\supseteq $ in the second statement of the lemma. If $x\in \wildn^{\lambda}(Y)\backslash X$, then $x\in \wildn^{\kappa}(Y)\backslash X=\coprod_{j\in\bbn}\wildn^{\kappa}(B_j)\backslash\{b_j\}$ for all $\kappa<\lambda$. Find $i$ with $x\in B_{i}\backslash\{b_i\}$. Then $x\in \wildn^{\kappa}(B_i)\backslash\{b_i\}$ for all $\kappa<\lambda$. Thus \[x\in \bigcap_{\kappa<\lambda}\left(\wildn^{\kappa}(B_i)\backslash\{b_i\}\right)= \left(\bigcap_{\kappa<\lambda}\wildn^{\kappa}(B_i)\right)\backslash\{b_i\}=\wildn^{\lambda}(B_i)\backslash\{b_i\}\] proving the inclusion $\subseteq$ of (2).

Suppose $\lambda=\kappa+1$. For the inclusion, fix $j\in\bbn$. We consider two cases:

\textbf{Case I:} Suppose $b_j\in \wildn^{\kappa}(Y)$. Since we have assumed $\wildn^{\kappa}(B_j)\subseteq \wildn^{\kappa}(Y)$ and $\wildn^{\kappa}(Y)\backslash X=\coprod_{j\in\bbn}\wildn^{\kappa}(B_j)\backslash\{b_j\}$, we have $\wildn^{\kappa}(Y)\cap B_j=\wildn^{\kappa}(B_j)\cup\{b_j\}$. Thus we may write $\wildn^{\kappa}(Y)$ uniquely as a one-point union $Z\vee \wildn^{\kappa}(B_j)$ with wedgepoint $b_j$. In particular, $\wildn^{\kappa}(B_j)$ is a retract of $\wildn^{\kappa}(Y)$ and it follows from Lemma \ref{inductionsteplemma} that $\wildn^{\lambda}(B_j)\subseteq \wildn^{\lambda}(Y)$. Thus (1) holds. For (2), the inclusion $\supseteq $ follows from (1). Suppose $x\in \wildn^{\lambda}(Y)\cap B_j\backslash X$. Find a fully essential map $f:(\bbe_n,b_0)\to (\wildn^{\kappa}(Y),x)$. Since $ B_j\backslash\{b_j\}$ is an open neighborhood of $x$ in $Y$, we may disregard finitely many spheres in $\bbe_n$ and assume that $f$ has image in $(B_j\backslash\{b_j\})\cap \wildn^{\kappa}(Y)$. Since $f$ has image in $\wildn^{\kappa}(Y)\backslash X=\coprod_{j\in\bbn}\wildn^{\kappa}(B_j)\backslash\{b_j\}$, $f$ must have image in $\wildn^{\kappa}(B_j)\backslash\{b_j\}$. If the $k$-th restriction of $f$ is not null-homotopic in $\wildn^{\kappa}(Y)$ it is not null-homotopic in the subspace $\wildn^{\kappa}(B_j)$. Since $f:(\bbe_n,b_0)\to (\wildn^{\kappa}(B_j),x)$ is fully essential, we have $x\in \wildn^{\lambda}(B_j)$. Thus $x\in \wildn^{\lambda}(B_j)\backslash \{b_j\}$, proving the inclusion $\subseteq$ for (2).

\textbf{Case II:} Suppose $b_j\notin \wildn^{\kappa}(Y)$. Now, the induction hypothesis $\wildn^{\kappa}(B_j)\subseteq \wildn^{\kappa}(Y)$ implies that $b_j\notin  \wildn^{\kappa}(B_j)$ and \[\wildn^{\kappa}(Y)=\wildn^{\kappa}(B_j)\sqcup\wildn^{\kappa}(Y)\backslash\wildn^{\kappa}(B_j).\]
It follows that \[\wildn^{\lambda}(Y)=\wildn^{\lambda}(B_j)\sqcup\wildn(\wildn^{\kappa}(Y)\backslash\wildn^{\kappa}(B_j)).\]
In particular, $\wildn^{\lambda}(B_j)\subseteq \wildn^{\lambda}(Y)$, proving (1). The inclusion $\supseteq$ for (2) follows again from (1). Let $x\in \wildn^{\lambda}(Y)\backslash X$. Find $i$ with $x\in B_i\backslash\{b_i\}$. Since \[\wildn(\wildn^{\kappa}(Y)\backslash\wildn^{\kappa}(B_i))\subseteq Y\backslash B_i,\] we must have $x\in \wildn^{\lambda}(B_i)=\wildn^{\lambda}(B_i)\backslash\{b_i\}$. This proves the inclusion $\subseteq$ for (2).

In Cases I and II, we have established both (1) and (2) in the statement of the lemma. This completes the induction.
\end{proof}

Note that the equality $\wildn^{\lambda}(Y)\cap B_j=\wildn^{\lambda}(B_j)$ need not be implied by Lemma \ref{swlemma} since we may have $b_j\in \wildn^{\lambda}(Y)$ for reasons unrelated to wildness internal to $B_j$. The next corollary introduces the hypothesis $b_j\in \wildn^{\lambda}(B_j)$ to obtain this conclusion. Since it follows directly from Lemma \ref{swlemma}, we omit the proof.

\begin{corollary}\label{sanwildsetcor}
Consider a shrinking point-attachment $Y=\san(X,A,\scrb)$ for Peano continuum $X$, sequence $A=\{a_j\}_{j\in\bbn}$ of distinct points in $X$, and sequence $\scrb=\{(B_j,b_j)\}_{j\in\bbn}$ of based spaces. Let $\lambda$ be an ordinal. If $b_j\in \wildn^{\lambda}(B_j)$ for all $j\in\bbn$, then $\wildn^{\lambda}(Y)\cap B_j=\wildn^{\lambda}(B_j)$ for all $j\in\bbn$.
\end{corollary}

\begin{lemma}\label{xinsanlemma}
Suppose $X$ is a separable, first countable, and locally path connected space. Let $A=\{a_j\}_{j\in\bbn}$ be a sequence with dense image in $X$, $\scrb=\{(B_j,b_j)\}_{j\in\bbn}$ be a sequence of based spaces, and $Y=\san(X,A,\scrb)$. Let $\lambda$ be an ordinal. If $b_j\in \wildn^{\lambda}(B_j)$ for all $j\in\bbn$, then $\wildn^{\lambda}(Y)=X\cup \bigcup_{j\in\bbn}\wildn^{\lambda}(B_j)$.
\end{lemma}

\begin{proof}
By Corollary \ref{sanwildsetcor}, we have $\wildn^{\lambda}(Y)\cap B_j=\wildn^{\lambda}(B_j)$ for all $j\in\bbn$.. Therefore, it suffices to show that $X\subseteq \wildn^{\lambda}(Y)$ for all ordinals $\lambda$. We prove this inclusion by induction on $\lambda$. Certainly, it holds for $\lambda=0$. Suppose $X\subseteq \wildn^{\kappa}(Y)$ for all ordinals $\kappa< \lambda$. If $\lambda$ is a limit ordinal, then $X\subseteq \bigcap_{\kappa<\lambda}\wildn^{\kappa}(Y)=\wildn^{\lambda}(Y)$. 

Suppose $\lambda=\kappa+1$. We have assumed $X\subseteq \wildn^{\kappa}(Y)$ and will show that $X\subseteq \wildn^{\lambda}(Y)$. Note that $\im(A)\subseteq \bigcup_{k}\wildn^{\lambda}(B_j)\subseteq \wildn^{\lambda}(Y)$ where the second inequality follows from Part (1) of Lemma \ref{swlemma}. Suppose $x\in X\backslash\im(A)$. Since $\im(A)$ is dense in $X$, find a sequence $j_1<j_2<j_3<\cdots$ of integers such that $\{a_{j_i}\}\to x$ in $X$. For all $i\in\bbn$, we have $b_{j_i}\in \wildn^{\lambda}(B_{j_i})$ and so there exists fully essential maps $g_{i}:(\bbe_n,b_0)\to (\wildn^{\kappa}(B_{j_i}),b_{j_i})$. Since each $\wildn^{\kappa}(B_{j_i})$ contains $b_{j_i}$, $\wildn^{\kappa}(B_{j_i})$ is a retract of $\wildn^{\kappa}(Y)=X\cup \bigcup_{j\in\bbn}\wildn^{\kappa}(B_{j})$. Thus the $p$-th restriction $g_{i,p}$ of $g_i$ is not null-homotopic in $\wildn^{\kappa}(Y)$. Since $X$ is first countable and locally path connected, we may replace $\{j_i\}$ with a cofinal subsequence so that there exists a path $\alpha_i$ from $x$ to $a_{j_i}$ such that $\{\alpha_i\}_{i\in\bbn}$ converges to $x$. For an arbitrary neighborhood $U$ of $x$ in $Y$, we have $\im(\alpha_i)\subseteq U$ for all but finitely many $i$ and the definition of shrinking point-attachments gives $\im(g_i)\subseteq\wildn^{\kappa}(B_{j_i})\subseteq B_{j_i}\subseteq U$ for all but finitely many $i$. Hence, the sequence of path-conjugates $\{\alpha_i\ast g_{j_i,1}\}_{i\in\bbn}$ converges to $x$ in $Y$. Since $\alpha_i$ and $g_i$ have image in $\wildn^{\kappa}(Y)$, the sequence $\{\alpha_i\ast g_{j_i,1}\}_{i\in\bbn}$ also converges to $x$ in $\wildn^{\kappa}(Y)$. Moreover, since $g_{j_i,1}$ is essential in $\wildn^{\kappa}(Y)$, so is $\alpha_i\ast g_{j_i,1}$. Thus $x\in \wildn^{\kappa+1}(Y)=\wildn^{\lambda}(Y)$, completing the proof that $X\subseteq \wildn^{\lambda}(Y)$.
\end{proof} 

\begin{lemma}\label{shrinkingwedgeinductionlemma}
Suppose $\lambda_1<\lambda_2<\lambda_3<\cdots$ is a sequence of countable ordinals and $\lambda=\sup\{\lambda_j\mid j\in\bbn\}$. If, for each $j\in\bbn$, $(B_j,b_j)$ is a based space satisfying $\wildn^{\lambda_{j}}(B_j)=\{b_j\}$, then $\wildn^{\lambda}\left(\sw_{j\in\bbn}B_j\right)=\{x_0\}$ where $x_0$ is the wedgepoint of $\sw_{j\in\bbn}B_j$.
\end{lemma}

\begin{proof}
We will apply Lemma \ref{swlemma} by regarding $\sw_{j\in\bbn}B_j$ as a shrinking point-attachment space $\san(X,A,\scrb)$ with $X=\{x_0\}$ and $\scrb=\{(B_j,b_j)\}_{j\in\bbn}$. 

By Part (1) of Lemma \ref{swlemma}, we have $\bigcup_{j\in\bbn}\wildn^{\kappa}(B_j)\subseteq \wildn^{\kappa}(\sw_{j\in\bbn}B_j)$ for all $\kappa\leq \lambda$. In the case where $\kappa=\lambda_j$, we have $x_0=b_j\in \wildn^{\lambda_{j}}(B_j)\subseteq \wildn^{\lambda_j}(\sw_{j\in\bbn}B_j)$. Thus $x_0\in \bigcap_{j\in\bbn}\wildn^{\lambda_j}(\sw_{j\in\bbn}B_j)=\wildn^{\lambda}(\sw_{j\in\bbn}B_j)$. 

Note that $\wildn^{\lambda_{j+1}}(B_j)=\emptyset$ where $\lambda_{j}+1<\lambda$. By Part (2) of Lemma \ref{swlemma}, we have \[\wildn^{\lambda}\left(\sw_{j\in\bbn}B_j\right)\backslash\{x_0\}=\coprod_{j\in\bbn}\wildn^{\lambda}(B_j)\backslash\{b_j\}\subseteq \coprod_{j\in\bbn}\wildn^{\lambda_j+1}(B_j)\backslash\{b_j\}=\emptyset.\]
Hence, $x_0$ is the only point of $\wildn^{\lambda}(\sw_{j\in\bbn}B_j)$.
\end{proof}

Finally, we prove our main result. Recall the statement of Theorem \ref{thm1} from the introduction.

\begin{proof}[Proof of Theorem \ref{thm1}]
For the case $\lambda=1$, take $X_1=S^1$. For $\lambda\geq 2$, we proceed by induction on $\lambda<\omega_1$. For our base case $\lambda=2$, we take $X_2=\bbe_n$. Suppose the theorem holds (including the secondary statement) for all countable ordinals $\kappa$ with $2\leq \kappa<\lambda$. 

{\bf Case I:} Suppose $\lambda=\kappa+1$ where $\kappa$ is a successor ordinal. By our induction hypothesis (particularly, the secondary statement), we have an $n$-dimensional Peano continuum $(X_{\kappa},x_{\kappa})$ such that $\wildn^{\kappa-1}(X_{\kappa})=\{x_{\kappa}\}$. Let $A=\{a_j\}_{j\in\bbn}$ be a sequence of distinct points in $\bbe_n$ that enumerates a dense set in $\bbe_n$ and let $\scrb=\{(B_j,b_j)\}_{j\in\bbn}$ be the constant sequence of based spaces $(X_{\kappa},x_{\kappa})$ (we identify $b_j$ with $x_{\lambda}$). Let $X_{\lambda}=\san(\bbe_n,A,\scrb)$ be the shrinking point-attachment space where a copy of $X_{\kappa}$ is attached to $\bbe_n$ at each point of $A$ and let $x_{\lambda}$ denote the image of $b_0\in\bbe_n$ in $X_{\lambda}$. By Lemma \ref{peanoprop}, $X_{\lambda}$ is an $n$-dimensional Peano continuum. 

Since $x_{\kappa}\in \wildn^{\kappa-1}(X_{\kappa})$, we have $b_j\in \wildn^{\kappa-1}(B_j)$ for each $j\in\bbn$. Thus $\wild_{n}^{\kappa-1}(X_{\lambda})\cap B_j=\wild_{n}^{\kappa-1}(B_j)=\{b_j\}$ for each $j\in\bbn$. Also, since $A$ is dense in $\bbe_n$, Lemma \ref{xinsanlemma} gives $\bbe_n\subseteq\wild_{n}^{\kappa-1}(\san(\bbe_n,A,\scrb))$. Thus $\wild_{n}^{\kappa-1}(X_{\lambda})=\bbe_n$. Applying the wild set to the equality $\wild_{n}^{\kappa-1}(X_{\lambda})=\bbe_n$ twice, we have $\wild_{n}^{\kappa}(X_{\lambda})=\{x_{\lambda}\}$ and $\wild_{n}^{\lambda}(X_{\lambda})=\emptyset$, which is the desired result.

{\bf Case II:} Suppose $\lambda=\kappa+1$ where $\kappa$ is a limit ordinal. Find successor ordinals $\kappa_1<\kappa_2<\kappa_3<\cdots$ such that $\sup\{\kappa_i\mid i\in\bbn\}=\kappa$. By hypothesis, we can find based $n$-dimensional Peano continuum $(X_{\kappa_i},x_{\kappa_i})$ such that $\wild_{n}^{\kappa_i-1}(X_{\kappa_i})=\{x_{\kappa_i}\}$. Let $X_{\lambda}=\sw_{i\in\bbn}X_{\kappa_i}$ with wedgepoint $x_{\lambda}$. since $\kappa=\sup\{\kappa_i-1\mid i\in\bbn\}$, by Lemma \ref{shrinkingwedgeinductionlemma}, we have $\wildn^{\kappa}(X_{\lambda})=\wildn^{\lambda-1}(X_{\lambda})=\{x_{\lambda}\}$.

 {\bf Case III:} Suppose $\lambda$ is a limit ordinal and $\lambda>\omega$. Find a sequence $2\leq\kappa_1<\kappa_2<\kappa_3<\cdots<\lambda$ of countable successor ordinals with $\sup\{\kappa_i\mid i\in\bbn\}=\lambda$. Using the induction hypothesis, find $n$-dimensional Peano continua $X_1,X_2,X_3,\dots$ such that $\rkn(X_i)=\kappa_i$ and $\wild_{n}^{\kappa_i}(X_{i})=\emptyset$ for each $i\in\bbn$. Pick a point $x_i\in X_{i}$. Attach a 1-cell $e_{i}=[0,1]$ by identifying $1\sim x_i$. Let $Y_i$ denote the resulting space and let $y_i$ be the image of $0$, i.e. the endpoint of the attached arc. Note that $Y_i$ is still an $n$-dimensional Peano continuum and $\wildn(Y_i)=\wildn(X_{i})$. Let $X_{\lambda}=\sw_{i\in\bbn}(Y_i,y_i)$. First, we consider what happens to a single application of $\wildn$ to $X_{\lambda}$. Each $Y_i$ is a retract of $X_i$ and since $\kappa_i\geq 2$, each $Y_i$ is not $n$-connected. Thus $x_0\in \wildn(X_{\lambda})$ and it follows that $\wildn(X_{\lambda})=\{x_0\}\sqcup \bigcup_{i\in\bbn}\wildn(X_{i})$. Note that each $\wildn(X_{i})$ is a union of path components of $\wildn(X_{\lambda})$. Thus $\wildn^2(X_{\lambda})=\coprod_{i\in\bbn}\wildn^2(X_{i})$. Applying \ref{disjointunionranklemma}, it follows that $\wild^{\lambda}(X_{\lambda})=\emptyset$ and $\wild^{\kappa}(X_{\lambda})\neq \emptyset$ whenever $\kappa<\lambda$. Thus $\rkn(X_{\lambda})=\lambda$.
\end{proof}

The proof of our second main result has a similar structure but the details in each case differ from the proof of Theorem \ref{thm1}.

\begin{proof}[Proof of Theorem \ref{thm2}]
The first direction follows from Lemma \ref{countablefreerank}. For the other direction, we proceed by transfinite induction. For $\lambda=0$, we consider any perfectly free $\pi_n$-wild space $X_0$ (recall that a perfectly $\pi_n$-wild Peano continuum is such a space and these are shown to exist in \cite{BrazasMitraHigher}). For $\lambda=1$, we take $X_1=S^n$. For the induction, suppose $\lambda$ is a countable ordinal and that we have constructed a continuum $X_{\kappa}$ with $\frkn(X_{\kappa})=\kappa$ for all $\kappa<\lambda$. Moreover, we assume that $\fwildn^{\kappa-1}(X_{\kappa})=S^n$ (and $\fwildn^{\kappa}(X_{\kappa})=\emptyset$) whenever $\kappa$ is a successor ordinal. In contrast, whenever $\kappa<\lambda$ is a limit ordinal it must necessarily be the case that $\fwildn^{\kappa}(X_{\kappa})$ is perfectly free $\pi_n$-wild since a descending sequence of closed non-empty sets in a compact metric space must have non-empty intersection.

{\bf Case I}: Suppose $\lambda=\kappa+1$ where $\kappa$ is a successor ordinal. By our induction hypothesis, we have continuum $X_{\kappa}$ with $\frkn(X_{\kappa})=\kappa$ and $\fwildn^{\kappa-1}(X_{\kappa})=S^n$. Let $X_{\kappa}^{+}=X_{\kappa}\sqcup\{\ast\}$ be the space $X_{\kappa}$ with an additional isolated basepoint. Let $A=\{a_j\}_{j\in\bbn}$ be a sequence in $S^n$ with dense image and let $\scrb$ be the constant sequence of based spaces $(X_{\kappa}^{+},\ast), (X_{\kappa}^{+},\ast), (X_{\kappa}^{+},\ast), \dots $. Set $X_{\lambda}=\san(S^n,A,\scrb)$. By Proposition \ref{peanoprop}, $X_{\lambda}$ is a continuum. Note that $\fwildn^{\kappa-1}(X_{\lambda})\backslash S^n$ is a disjoint union of $n$-spheres of null-diameters that limit on $S^n$. It follows that $\fwildn^{\kappa}(X_{\lambda})=S^n$. This gives $\frkn(X_{\lambda})=\kappa+1=\lambda$ and $\fwildn^{\lambda-1}(X_{\lambda})=S^n$.

{\bf Case II}: Suppose $\lambda=\kappa+1$ where $\kappa$ is a limit ordinal. Find successor ordinals $\kappa_1<\kappa_2<\kappa_3<\cdots$ such that $\sup(\kappa_j)=\kappa$. Let $X_{\kappa_j}$ be a continuum with $\frkn(X_{\kappa_j})=\kappa_j$ and $\fwildn^{\kappa_j}(X_{\kappa_j})=\emptyset$. Let $X_{\kappa_j}^{+}=X_{\kappa_j}\sqcup \{\ast_j\}$ have an isolated basepoint and define $X_{\lambda}=\sw_{j\in\bbn}X_{\kappa_j}^{+}$ with wedgepoint $x_{\lambda}$. We have $\fwildn^{\kappa_j}(X_{\lambda})=\sw_{k>j}X_{\kappa_k}^{+}$ for all $j\in\bbn$. Then $\fwildn^{\kappa}(X_{\lambda})=\bigcap_{j\in\bbn}\left(\sw_{k>j}X_{\kappa_k}^{+}\right)=\{x_0\}$. Thus $\fwildn^{\lambda}(X_{\lambda})=\emptyset$ and $\frkn(X_{\lambda})=\lambda$.

{\bf Case III}: Suppose $\lambda$ is a limit ordinal. Let $P$ be a perfectly free $\pi_n$-wild continuum. Find ordinals $\kappa_1<\kappa_2<\kappa_3<\cdots<\lambda $ such that $\sup(\kappa_j)=\lambda$. Fix continuum $X_{\kappa_j}$ with $\frkn(X_{\kappa_j})=\kappa_j$ and let $X_{\kappa}^{+}$ be the space with an additional isolated basepoint. Let $W=\sw_{j\in\bbn}(X_{\kappa_j}^{+})$ with wedgepoint $w_0$.

Let $A=\{a_j\}_{j\in\bbn}$ be a sequence with dense image in $P$ and let $\scrb$ be the constant sequence of based spaces $(W,w_0),(W,w_0),(W,w_0),\dots$. Set $X_{\lambda}=\san(P,A,\scrb)$ and note that $X_{\lambda}$ is a continuum by Proposition \ref{peanoprop}. By construction, $X_{\lambda}\backslash P$ is a topological disjoint union of a countably infinite number copies of each space $X_{\kappa_j}$. Hence, $\fwildn^{\kappa}(X_{\lambda})\backslash P\neq \emptyset $ for all $\kappa<\lambda$ and $\fwildn^{\lambda}(X_{\lambda})=P$. Since $\fwild_n(P)=P$, we have $\frkn(X_{\lambda})=\lambda$.

\end{proof}

\begin{example}\label{differenceexample}
With Theorems \ref{thm1} and \ref{thm2} established, we show that it is possible for $\rkn(Z)=\frkn(Z)+\lambda$ to occur where the difference $\lambda$ is an arbitrary infinite countable ordinal.

Let $X$ be an $n$-dimensional Peano continuum with $\rkn(X)=\lambda$. By standard embedding theorems, we may view $K$ as a subspace of $[0,1]^{k}$ for some $k\leq 2n+1$. Let $A=\{a_j\}_{j\in\bbn}$ be a sequence in $[0,1]^k$ where $\{a_{2i-1}\mid i\in\bbn\}$ is a dense subset of $K$ and $\{a_{2i}\mid i\in \bbn\}$ is a countable dense subset of $[0,1]^k$. Let $\scrb$ be the alternating sequence of based spaces $(\bbe_n,b_0),(\fe_n,a_0),(\bbe_n,b_0),(\fe_n,a_0),\dots$ and set $Y=\scrs(X,A,\scrb)$. Then $Y$ is an $n$-dimensional continuum with $\wildn(Y)=K$ and $\fwildn(Y)=[0,1]^k$. By Theorem \ref{realizingcompactmetricspacesthm}, there exists an $n$-dimensional Peano continuum $Z$ such that $\wildn(Z)=Y$ (we also have $\fwildn(Z)=Y$ by Lemma \ref{comparelemma}). 

We have $\rkn(Z)=2+\rkn(X)=2+\lambda=\lambda$ since $\lambda$ is infinite. On the other hand, $\fwildn^2(Z)=[0,1]^k$ and thus $\frkn(Z)=3$. Hence, $\rkn(Z)=\lambda=3+\lambda=\frkn(Z)+\lambda$.
\end{example}

We conclude with an observation regarding Problem \ref{openprob2}. The following proposition suggests that, even in the one-dimensional setting, a (Peano) continuum with uncountable $\pi_1$-wild rank will be fairly complicated. In particular, one should only attempt to construct a perfectly free $\pi_1$-wild example. The authors acknowledge that the existence of such a space may be independent of ZFC.

\begin{proposition}
If there exists a one-dimensional continuum $X$ with $\rk_1(X)\geq \omega_1$, then there exists a perfectly free $\pi_1$-wild subcontinuum $Y\subseteq X$ with $\rk_1(Y)\geq \omega_1$.
\end{proposition}

\begin{proof}
Since $X$ is a one-dimensional continuum, $\frk_1(X)=\lambda$ is countable by Theorem \ref{countablefreerank}. Let $Y=\fwild_{1}^{\lambda}(X)$. By Proposition \ref{onedimprop}, we have $\wild_{1}^{\kappa}(X)\subseteq \fwild_{1}^{\kappa}(X)$ for all ordinals $\kappa$. Thus $\wild_{1}^{\lambda}(X)\subseteq \fwild_{1}^{\lambda}(X)$. Since $\rk_1(X)$ is not countable, $\wild_{1}^{\lambda}(X)\neq \emptyset$. It follows that $Y$ is a non-empty perfectly free $\pi_1$-wild space. By Lemma \ref{initialsumrank}, we have $\rk_1(X)=\lambda+\rk_1(Y)$. Since $\rk_1(X)$ is uncountable, so is $\rk_1(Y)$.
\end{proof}

\end{document}